\documentclass[11pt]{article}
\usepackage{amsfonts}
\usepackage{graphicx}
\usepackage{tabularx}
\usepackage{array}
\usepackage[usenames,dvipsnames]{color}
\usepackage{comment}
\usepackage{amsmath}
\usepackage{amsthm}
\usepackage{amssymb}
\usepackage{fullpage}
\usepackage[dvipsnames]{xcolor}
\usepackage{tikz}
\usepackage{listings}

\DeclareMathOperator{\sv}{\textsc{Save}}

\DeclareMathOperator{\ww}{wd}

\newcommand{\ds}{\textsf{DelSave}}

\newtheorem{theorem}{Theorem}[section]

\newtheorem{question}[theorem]{Question}
\newtheorem{claim}[theorem]{Claim}
\newtheorem{lemma}[theorem]{Lemma}
\newtheorem{corollary}[theorem]{Corollary}
\theoremstyle{definition}
\newtheorem{definition}[theorem]{Definition}

\newtheorem*{theorem13}{Theorem \ref{thm:AT}}
\newtheorem*{theorem14}{Theorem \ref{thm:chromatic_intro}}

\def\epsilon{\varepsilon}

\title{Paintability of $r$-chromatic graphs}
\author{Peter Bradshaw\thanks{Department of Mathematics, University of Illinois Urbana-Champaign. Peter Bradshaw received funding from NSF RTG grant DMS-1937241.} \
and Jinghan A Zeng\thanks{Department of Mathematics, University of Illinois Urbana-Champaign. Jinghan Zeng received funding from NSF grant DMS-1937241 through the ICLUE program of University of Illinois Urbana-Champaign.}}

\date{}
\begin{document}
\maketitle
\begin{abstract}
The online list coloring game is a two-player graph-coloring game
played on a graph $G$ as follows.
On each turn, a Lister reveals a new color $c$ at some subset $S \subseteq V(G)$ of uncolored vertices, and then a Painter chooses an independent subset of $S$ to which to assign $c$.
As the game is played, the revealed colors at each vertex $v \in V(G)$ form a color set $L(v)$, often called a list. 
The \emph{paintability} of $G$ measures the minimum value $k$ for which Painter has a strategy to complete a coloring of $G$ in such a way that $|L(v)| \leq k$ for each vertex $v \in V(G)$.
The paintability of a graph is an upper bound for its list chromatic number, or choosability. 

The online list coloring game is a special case of the \emph{DP-painting} game, which is defined similarly using the setting of DP-coloring. 
In the DP-painting game, the Lister reveals correspondence covers of a graph $G$ rather than colors, and the Painter chooses independent subsets of these covers.
The DP-painting game has a parameter known as \emph{DP-paintability} which is analogous to paintability.

In this paper, we consider upper bounds for the paintability and DP-paintability of a graph $G$ with large maximum degree $\Delta$ and chromatic number at most some fixed value $r$.
We prove that the paintability of $G$ is at most $\left(1 - \frac{1}{4r+1} \right ) \Delta + 2$ and that the DP-paintability of $G$ is at most $\Delta - \Omega( \sqrt{\Delta \log \Delta})$. We prove our first upper bound using Alon-Tarsi orientations, and we prove our second upper bound by considering the \emph{strict type-$3$ degeneracy} parameter recently introduced by Zhou, Zhu, and Zhu.
\end{abstract}
\section{Introduction}

\subsection{Background: Online list coloring and online DP-coloring}
Given a graph $G$, a \emph{list assignment} on $G$ is a function $L:V(G) \rightarrow 2^{\mathbb N}$ that assigns a color set $L(v)$ to each vertex $v \in V(G)$.
Given a function $f:V(G) \rightarrow \mathbb N$, a list assignment $L$ which satisfies $|L(v)| = f(v)$ for all $v \in V(G)$ is called an \emph{$f$-assignment}, and $L$ is called a \emph{$k$-assignment} if $|L(v)| = k$ for all $v \in V(G)$.
Given a list assignment $L$ on $G$, an \emph{$L$-coloring} of $G$ is a mapping $\phi:V(G) \rightarrow \mathbb N$ such that $\phi(u) \neq \phi(v)$ for each adjacent vertex pair $uv \in E(G)$ and such that $\phi(v) \in L(v)$ for each vertex $v \in V(G)$. The graph $G$ is \emph{$L$-colorable} if $G$ has an $L$-coloring, and  $G$ is \emph{$k$-choosable} if $G$ has an $L$-coloring for every $k$-assignment $L$.
The problem of determining whether a graph $G$ has an $L$-coloring for a given list assignment $L$, introduced independently by Vizing \cite{Vizing} and Erd\H{o}s, Rubin, and Taylor \cite{ERT}, is called the \emph{list coloring problem}.
The minimum value $k$ for which $G$ is $k$-choosable is called the \emph{choosability} of $G$ and is denoted by $ch(G)$. 

The list coloring problem has the following online variant, known as the  \emph{painting game}, introduced independently by Schauz \cite{SchauzP} and Zhu \cite{ZhuP}. Given a graph $G$ and a mapping $f:V(G) \rightarrow \mathbb N$, two players known as Lister and Painter play the \emph{$f$-painting game}
as follows. Each vertex $v \in V(G)$ begins with $f(v)$ tokens. On each turn $i$, Lister
removes at most one token from each uncolored vertex of $G$ and
denotes by $S_i$ the set of vertices from which a token has been removed on turn $i$. Then, Lister reveals the color $i$ at each vertex of $S_i$, and Painter colors an independent subset of $S_i$ with the color $i$. The game ends when either every vertex of $G$ is colored, or when at least one uncolored vertex of $G$ has zero tokens at the end of a turn. Painter wins the game if every vertex of $G$ has been colored at the end of the game; otherwise, Lister wins. We say $G$ is \emph{$f$-paintable} if Painter has a winning strategy in the $f$-painting game. We write $\chi_P(G)$ for the \emph{paintability} of $G$, which is the minimum integer $k$ such that $G$ is $f$-paintable for the constant function $f(v) = k$. 
Schauz \cite{SchauzP} and Zhu \cite{ZhuP} observed that for every graph $G$, $ch(G) \leq \chi_P(G)$.

%Throughout the $f$-painting game on a graph $G$, the $f(v)$ colors revealed at each vertex $v \in V(G)$ can be viewed as a color list $L(v)$, and hence Lister's strategy gives an $f$-assignment $L$ on $G$. As Painter wins the $f$-painting game only if $G$ is $L$-colorable, and as Painter is able to produce any list assignment $L$ up to relabelling of colors, it follows that $\ch(G) \leq \chi_P(G)$. Furthermore, as the $f$-painting game is a version of the DP-$f$-painting game in which Lister's strategy is restricted, it follows that $\chi_P(G) \leq \chi_{DPP}(G)$.

%Given a graph $G$ and a list assignment $L:V(G) \rightarrow 2^{\mathbb N}$, the problem of deciding whether $G$ is $L$-colorable can be equivalently phrased as follows. Let $H$ be a graph whose vertex set is obtained by taking a disjoint union of the lists $L(v)$ over $v \in V(G)$. Then, for each edge $uv \in E(G)$ and each common color $c \in L(u) \cap L(v)$, add an edge in $H$ joining the vertex in $L(u)$ corresponding with the color $c$ to the vertex in $L(v)$ corresponding with the color $c$. Then, an $L$-coloring of $G$ is equivalent to an independent set in $H$ that includes exactly one vertex from each set $L(v)$.

Kim, Kostochka, Li, and Zhu \cite{KKLZ} introduced a generalized version of the painting game that serves as an online version of the \emph{DP-coloring} problem of Dvo\v{r}\'ak and Postle.
Given a graph $G$ and a mapping $f:V(G) \rightarrow \mathbb N$,
the \emph{DP-$f$-painting game} is played as follows.
Each vertex $v \in V(G)$ begins with $f(v)$ tokens, and each vertex of $G$ begins the game unmarked.
On each turn, Lister removes some number $g(v) \geq 0$ of tokens from each vertex $v \in V(G)$. Lister then generates a graph $H$ and a function $L:V(G) \rightarrow 2^{V(H)}$ which obeys the following rules:
\begin{itemize}
\item The sets $L(v)$ taken over $v \in V(G)$ form a partition of $V(H)$. 
\item For each $v \in V(G)$, $|L(v)|=g(v)$.
\item If $a,b \in L(v)$ for some $v \in V(G)$, then $a$ and $b$ are adjacent in $H$. 
\item Given two distinct vertices $u,v \in V(G)$, if $a \in L(u)$ and $b \in L(v)$ are adjacent in $H$, then $uv \in E(G)$.
\item If $u$ and $v$ are adjacent in $G$, then the edges in $H$ between $L(u)$ and $L(v)$ form a (possibly empty) matching. 
\end{itemize}

\begin{figure}
        \begin{center}
    \begin{tikzpicture}
[scale=1,auto=left,every node/.style={circle,fill=gray!30,minimum size = 6pt,inner sep=0pt}]

\node(z) at (0,-0.4) [draw=white,fill=white]{$2$};
\node(z) at (-1.4,1) [draw=white,fill=white]{$2$};
\node(z) at (0,2.4) [draw=white,fill=white]{$3$};
\node(z) at (1.4,1) [draw=white,fill=white]{$3$};

\node(v1) at (0,0) [draw = black] {};
\node(w1) at (1,1) [draw = black] {};
\node(x1) at (0,2) [draw = black] {};
\node(y1) at (-1,1) [draw = black] {};
\node(z) at (2,-0.5) [draw=white,fill=white]{};

\draw   (v1) -- (w1);
\draw   (x1) -- (w1);
\draw (y1) -- (x1);
\draw   (y1) -- (v1);

\end{tikzpicture}
    \begin{tikzpicture}
[scale=1,auto=left,every node/.style={circle,fill=gray!30,minimum size = 6pt,inner sep=0pt}]
\node(z) at (-1,0) [draw=white,fill=white] {};
\node(v1) at (0,0) [draw = black] {};
\node(v2) at (0,0.5) [draw = black] {};

\node(w1) at (1,1) [draw = black] {};
\node(w2) at (1,1.5) [draw = black] {};
\node(w3) at (1,2) [draw = black] {};

\node(x1) at (0,2) [draw = black] {};
\node(x2) at (0,2.5) [draw = black] {};
\node(x3) at (0,3) [draw = black] {};

\node(y1) at (-1,1) [draw = black] {};
\node(y2) at (-1,1.5) [draw = black] {};

\draw[draw=black] (-0.3,-0.3) rectangle ++(0.6,1.1);
\draw[draw=black] (-0.3,1.7) rectangle ++(0.6,1.6);
\draw[draw=black] (-1.3,0.7) rectangle ++(0.6,1.1);
\draw[draw=black] (0.7,0.7) rectangle ++(0.6,1.6);

\draw (w1) -- (w2);
\draw (w3) -- (w2);
  \path[-]          (w1)  edge   [bend right=30]   (w3);

\draw (x1) -- (x2);
\draw (x3) -- (x2);
  \path[-]          (x1)  edge   [bend left=30]   (x3);

\draw (w3) -- (x3);
\draw  (v1) -- (w1);
\draw   (x1) -- (w1);
\draw   (y1) -- (x1);
\draw   (y1) -- (v2);
\draw   (v2) -- (w2);
\draw   (x2) -- (w2);
\draw   (y2) -- (x2);
\draw   (y2) -- (v1);
\draw   (y2) -- (y1);
\draw   (v2) -- (v1);

\end{tikzpicture}
\end{center}
\caption{The figure on the left shows a graph $G$ along with a function $g:V(G) \rightarrow \mathbb N$. The figure on the right shows a $g$-cover of $G$.}
\label{fig:DP}
\end{figure}

The pair $(H,L)$ that Lister generates is called a \emph{$g$-cover} of $G$. An example of a $g$-cover is shown in Figure \ref{fig:DP}.
After Lister generates $H$ and $L$, Painter responds by selecting an independent set $I$ in $H$. For each $v \in V(G)$, if $L(v) \cap I$ is nonempty, then Painter marks the vertex $v$. The game ends when all vertices 
of $G$
run out of tokens. 
We require that Lister never lets the function $g$ be identically zero, so that at least one token is removed on each turn and the game ends in finite time.
Painter wins if all vertices are marked at the end of the game; otherwise, Lister wins. A graph is \emph{DP-$f$-paintable} if Painter has a winning strategy in the DP-$f$-painting game.
We write $\chi_{DPP}(G)$ for the \emph{DP-paintability} of $G$, which is the minimum integer $k$ such that $G$ is DP-$f$-paintable for the constant function $f(v) = k$.

%If Lister chooses to remove exactly $f(v)$ tokens from each vertex $v \in V(G)$ on the first turn, so that the game only lasts one turn, then the question of whether Painter can win the game
%is equivalent to the original DP-coloring problem of Dvo\v{r}\'ak and Postle \cite{DP}. 
%Furthermore, if Lister chooses to remove at most one token from each vertex on each turn, so that $g(v) \leq 1$ for each $v \in V(G)$ on every turn, then the game is equivalent to the \emph{online list coloring game}.

%., giving the inequality chain 
%\[\ch(G) \leq \chi_P(G) \leq \chi_{DPP}(G).\]

\subsection{Background: Alon-Tarsi number}
%If the edges of a graph $G$ have an acyclic orientation in which each vertex $v \in V(G)$ has at most $k$ out-neighbors, and if $L$ is a $(k+1)$-assignment on $G$, then $G$ is $L$-colorable by the following argument.
%We greedily color the vertices of $G$ one at a time so that each vertex is colored before its in-neighbors, and we observe that on each step, the vertex that we color has at least one available color. 
%The minimum value $k$ for which $G$ has an acyclic orientation with maximum out-degree $k$ is the \emph{degeneracy} of $G$, written $\degen(G)$.
%The greedy algorithm described above implies that for every graph $G$, $\ch(G)  \leq \degen(G) + 1$.

Alon and Tarsi \cite{AT} introduced the following sufficient condition for showing that a graph is $f$-choosable for some function $f:V(G) \rightarrow \mathbb N$.
Given an oriented graph $\vec G$, a spanning
subgraph $\vec H \subseteq \vec G$ is
\emph{Eulerian} if $V(\vec H) = V(\vec G)$ and each vertex $v \in V(\vec H)$ satisfies $\deg^+_{\vec H}(v) = \deg^-_{\vec H}(v)$. Note that the edgeless subgraph is also Eulerian. We say that $\vec H$ is \emph{even} if $|E(\vec H)|$ is even; otherwise, $\vec H$ is \emph{odd}. We write $EE(\vec G)$ for the number of even Eulerian subgraphs $\vec H \subseteq \vec G$, and we write $EO(\vec G)$ for the number of odd Eulerian subgraphs of $\vec G$. 
Writing $G$ for the undirected graph that underlies $\vec G$, $\vec G$ is an \emph{Alon-Tarsi orientation} of $G$ if 
$EE(\vec G) \neq EO(\vec G)$. We show an example of an Alon-Tarsi orientation in Figure \ref{fig:AT}.

\begin{figure}
\begin{center}
\begin{tikzpicture}
[scale=1.2,auto=left,every node/.style={circle,fill=gray!30,minimum size = 6pt,inner sep=0pt}]
\node(z) at (-1.15,-0.25) [draw=white,fill=white] {$w$};
\node(z) at (0.9,0.5) [draw=white,fill=white] {$y$};
\node(z) at (0,-0.25) [draw=white,fill=white] {$z$};
\node(z) at (0,1.25) [draw=white,fill=white] {$v$};
\node(z) at (-1.15,1.25) [draw=white,fill=white] {$u$};
\node(z) at (-0.45,0.25) [draw=white,fill=white] {$x$};
\node(p1) at (-1,0) [draw = black] {};
\node(p2) at (-0.5,0.5) [draw = black] {};
\node(p3) at (0.6,0.5) [draw = black] {};
\node(p4) at (-1,1) [draw = black] {};
\node(r) at (0,0) [draw=black] {};
\node(s) at (0,1) [draw=black] {};
%\node(h1) at (-1,0) [draw=black] {};
%\node(h2) at (1,0) [draw = black] {};

\draw[->] (p1) -- (p2);
\draw[->] (p2) -- (p4);
\draw[->] (p4) -- (p1);
\draw[->] (p2) -- (p3);
\draw[->] (p3) -- (r);
\draw[->] (p4) -- (s);
\draw[->] (s) -- (p3);
\draw[->] (r) -- (p1);

\foreach \from/\to in {p1/p2,p2/p4,p2/p3,r/p1,p4/p1,r/p3,s/p3,s/p4}
    \draw (\from) -- (\to);
\end{tikzpicture}
\caption{The figure shows an oriented graph $\vec G$ with an Alon-Tarsi orientation. $\vec G$ has a single odd Eulerian subgraph $\{uw,wx,xu\}$, so $EO(\vec G) = 1$. Furthermore, $\vec G$ has three even Eulerian subgraphs, namely the edgeless subgraph, $\{wx,xy,yz,zw\}$, and $\{uv,vy,yz,zw,wx,xu\}$, so $EE(\vec G) = 3$. As $EE(\vec G) \neq EO(\vec G)$, the orientation of $\vec G$ is an Alon-Tarsi orientation.}
\label{fig:AT}
\end{center}
\end{figure}

Given an undirected graph $G$ and a function $f:V(G) \rightarrow \mathbb N$, $G$ is \emph{$f$-Alon-Tarsi} if $G$ has an Alon-Tarsi orientation $\vec G$ satisfying $\deg_{\vec G}^+(v) < f(v)$ for each $v \in V(G)$.
Then, the
\emph{Alon-Tarsi number} of $G$, written $AT(G)$, 
is the minimum value $k$ such that $G$ is $f$-Alon Tarsi for the constant function $f(v) = k$. 
%For every orientation $\vec G$ of a graph $G$, the edgeless subgraph of $\vec G$ is an even Eulerian subgraph, and so $|EE(\vec G)| \geq 1$. Furthermore, if $\vec G$ is an acyclic orientation, then $|EO(\vec G)| = 0$. Therefore, an acyclic orientation of $G$ is also an Alon-Tarsi orientation, and hence $AT(G) \leq \degen(G) + 1$. 
Alon and Tarsi \cite{AT} proved that $ch(G) \leq AT(G)$ for every graph $G$.
%giving the inequality chain
%\[\ch(G) \leq AT(G) \leq \degen(G) + 1.\]
Schauz \cite{SchauzAT} later proved the stronger inequality $\chi_P(G) \leq AT(G)$.

\subsection{Background: Strict type-$3$ and strict type-$4$ degeneracy}
All of the parameters of a graph $G$ introduced in the previous subsections, namely $ch(G)$, $\chi_P(G)$, $\chi_{DPP}(G)$, and $AT(G)$, are bounded above by the \emph{strict degeneracy} of a graph, written $sd(G)$ and defined as the minimum integer $d$ for which $G$ has an acyclic orientation in which each vertex $v \in V(G)$ has out-degree $\deg^+(v) < d$. 
The strict degeneracy of a graph is also called the \emph{coloring number}, and it is equal to one plus the \emph{degeneracy} of a graph.
Recently, Zhou, Zhu, and Zhu \cite{ZZZ} introduced four additional types of strict degeneracy with the aim of finding sharper upper bounds for the parameters introduced in the previous sections. In this paper we consider their strict degeneracy parameters of types $3$ and $4$.

The \emph{strict type-$3$ degeneracy} and \emph{strict type-$4$ degeneracy} of a graph $G$
are
defined using the following $\ds$ operation. The operation is defined in a restricted form in \cite{BL} and appears in the following general form in \cite{HHZ} and \cite{ZZZ}:
%We fully define weak degeneracy in Section \ref{sec:parameters}.
\begin{definition}
    Let $G$ be a graph, and let $f:V(G) \rightarrow \mathbb N$ be a positive-valued function. For a vertex $u \in V(G)$ and a subset $W \subseteq N(u)$,
    %of size at most $1$\footnote{In fact, Bernshteyn and Lee define two separate functions: one for the case that $W = \emptyset$, and one for the case that $|W| = 1$. We unify their two definitions for conciseness.}, 
    the operation $\ds(G,f,u,W)$ outputs the graph $G' = G - u$ and the mapping $f':V(G') \rightarrow \mathbb Z$ defined by 
    \[f'(v) = \begin{cases}
        f(v) - 1 & \textrm{ if } v \in N(u) \setminus W \\
        f(v) & \textrm{ otherwise}.
    \end{cases}
    \]
\end{definition}
Application of the operation $\ds(G,f,u,W)$ is \emph{legal} if 
\begin{itemize}
    \item $f(u) > \sum_{w \in W} f(w)$, and
    \item $f'(v) \geq 1$ for all $v \in V(G')$.
\end{itemize}
We show an example of an application of the $\ds$ function in Figure \ref{fig:DS}.
We often say that when $\ds(G,f,u,W)$ is applied, $u$ \emph{saves} each vertex in the set $W$.
The operation $\ds(G,f,u,W)$ is \emph{restricted} if $|W| \leq 1$, in which case $u$ saves at most one vertex.

\begin{figure}
\begin{center}
\begin{tikzpicture}
[scale=1.2,auto=left,every node/.style={circle,fill=gray!30,minimum size = 6pt,inner sep=0pt}]

\clip (-2,-1) rectangle + (3,2.5);

\node(z) at (-0.5,-0.75) [draw=white,fill=white] {$G$};
\node(z) at (0.9,-2.12) [draw=white,fill=white] {};
\node(z) at (-1.15,-0.25) [draw=white,fill=white] {$w$};
\node(z) at (0.9,0.25) [draw=white,fill=white] {};
%\node(z) at (0,-0.25) [draw=white,fill=white] {$z$};
\node(z) at (0,1.25) [draw=white,fill=white] {$v$};
\node(z) at (-1.15,1.25) [draw=white,fill=white] {$u$};
\node(z) at (-0.45,0.25) [draw=white,fill=white] {$x$};
\node(p1) at (-1,0) [draw = black] {};
\node(p2) at (-0.5,0.5) [draw = black] {};
\node(p3) at (0,0.5) [draw = black] {};
\node(p4) at (-1,1) [draw = black] {};
\node(r) at (0,0) [draw=black] {};
\node(s) at (0,1) [draw=black] {};
%\node(h1) at (-1,0) [draw=black] {};
%\node(h2) at (1,0) [draw = black] {};

\foreach \from/\to in {p1/p2,p2/p4,p2/p3,r/p1,p4/p1,r/p3,s/p3,s/p4}
    \draw (\from) -- (\to);
\end{tikzpicture}
\begin{tikzpicture}
[scale=1.2,auto=left,every node/.style={circle,fill=gray!30,minimum size = 6pt,inner sep=0pt}]

\clip (-2,-1) rectangle + (3,2.5);

\node(z) at (-0.5,-0.75) [draw=white,fill=white] {$f:V(G) \rightarrow \mathbb N$};
\node(z) at (0.9,-2.11) [draw=white,fill=white] {};
\node(z) at (-1.15,-0.25) [draw=white,fill=white] {$2$};
\node(z) at (0.9,0.25) [draw=white,fill=white] {};
\node(z) at (0.25,-0.25) [draw=white,fill=white] {$3$};
\node(z) at (0.25,0.5) [draw=white,fill=white] {$3$};
\node(z) at (0.25,1.25) [draw=white,fill=white] {$1$};
\node(z) at (-1.15,1.25) [draw=white,fill=white] {$4$};
\node(z) at (-0.45,0.25) [draw=white,fill=white] {$2$};
\node(p1) at (-1,0) [draw = black] {};
\node(p2) at (-0.5,0.5) [draw = black] {};
\node(p3) at (0,0.5) [draw = black] {};
\node(p4) at (-1,1) [draw = black] {};
\node(r) at (0,0) [draw=black] {};
\node(s) at (0,1) [draw=black] {};
%\node(h1) at (-1,0) [draw=black] {};
%\node(h2) at (1,0) [draw = black] {};

\foreach \from/\to in {p1/p2,p2/p4,p2/p3,r/p1,p4/p1,r/p3,s/p3,s/p4}
    \draw (\from) -- (\to);
\end{tikzpicture}
\begin{tikzpicture}
[scale=1.2,auto=left,every node/.style={circle,fill=gray!30,minimum size = 6pt,inner sep=0pt}]

\clip (-2,-1) rectangle + (3,2.5);

\node(z) at (-1.15,-0.25) [draw=white,fill=white] {$1$};
\node(z) at (0.9,-2.12) [draw=white,fill=white] {};
\node(z) at (-0.5,-0.75) [draw=white,fill=white] {$\ds(G,f,u,\{v\}$)};
\node(z) at (-1.15,-0.25) [draw=white,fill=white] {$1$};
\node(z) at (0.25,1.25) [draw=white,fill=white] {$1$};
\node(z) at (-0.45,0.25) [draw=white,fill=white] {$1$};
\node(z) at (0.25,0.5) [draw=white,fill=white] {$3$};
\node(p1) at (-1,0) [draw = black,fill=black] {};
\node(p2) at (-0.5,0.5) [draw = black,fill=black] {};
\node(p3) at (0,0.5) [draw = black] {};
\node(z) at (0.25,-0.25) [draw=white,fill=white] {$3$};
\node(r) at (0,0) [draw=black] {};
\node(s) at (0,1) [draw=black] {};
%\node(h1) at (-1,0) [draw=black] {};
%\node(h2) at (1,0) [draw = black] {};

\foreach \from/\to in {p1/p2,p2/p3,r/p1,r/p3,s/p3}
    \draw (\from) -- (\to);
\end{tikzpicture}
\begin{tikzpicture}
[scale=1.2,auto=left,every node/.style={circle,fill=gray!30,minimum size = 6pt,inner sep=0pt}]

\clip (-2.25,-1) rectangle + (3.5,2.5);

\node(z) at (-1.15,-0.25) [draw=white,fill=white] {$2$};
%\node(z) at (0.9,0.5) [draw=white,fill=white] {$y$};
%\node(z) at (0,-0.25) [draw=white,fill=white] {$z$};
\node(z) at (-0.5,-0.75) [draw=white,fill=white] {$\ds(G,f,u,\{v,w\}$)};
\node(z) at (0.25,1.25) [draw=white,fill=white] {$1$};
\node(z) at (-0.45,0.25) [draw=white,fill=white] {$1$};
\node(p1) at (-1,0) [draw = black] {};
\node(p2) at (-0.5,0.5) [draw = black,fill=black] {};
\node(p3) at (0,0.5) [draw = black] {};
\node(z) at (0.25,-0.25) [draw=white,fill=white] {$3$};
\node(z) at (0.25,0.5) [draw=white,fill=white] {$3$};
\node(r) at (0,0) [draw=black] {};
\node(s) at (0,1) [draw=black] {};
\node(z) at (-1.15,-0.25) [draw=white,fill=white] {$2$};
%\node(h1) at (-1,0) [draw=black] {};
%\node(h2) at (1,0) [draw = black] {};

\foreach \from/\to in {p1/p2,p2/p3,r/p1,r/p3,s/p3}
    \draw (\from) -- (\to);
\end{tikzpicture}
\caption{The first figure shows a graph $G$, and the second figure shows a function $f$ that assigns a positive integer to each vertex of $G$. The third figure shows the graph and mapping that result from $\ds(G,f,u,\{v\})$. The fourth figure shows the graph and mapping that result from $\ds(G,f,u,\{v,w\})$. In the third and fourth figures, each vertex whose corresponding value decreases as a result of the $\ds$ function is shaded.}
\label{fig:DS}
\end{center}
\end{figure}

Given a function $f:V(G) \rightarrow \mathbb N$,
$G$ is \emph{strict type-$3$ $f$-degenerate} ($ST^{(3)}$-$f$-degenerate) if whenever $G$ is initially equipped with the function $f$, all vertices of $G$ can be deleted through repeated applications of legal $\ds$ operations. Furthermore, $G$ is \emph{strict type-$4$ $f$-degenerate} ($ST^{(4)}$-$f$-degenerate) if whenever $G$ is initially equipped with the function $f$, all vertices of $G$ can be deleted through repeated applications of legal \emph{restricted} $\ds$ operations. 
For $i \in \{3,4\}$, the \emph{strict type-$i$ degeneracy} of $G$, written $sd^{(i)}(G)$, is the minimum integer $k$ such that $G$ is $ST^{(i)}$-$f$-degenerate for the constant function $f(v) = k$.
It is straightforward to show that 
$sd^{(3)}(G) \leq sd^{(4)}(G) \leq sd(G)$.

Bernshteyn and Lee \cite{BL} first investigated strict type-$4$ degeneracy
using the closely related notion of \emph{weak degeneracy}. They defined the weak degeneracy $\ww(G)$ of a graph $G$ as $\ww(G) = sd^{(4)}(G) - 1$, and they  showed that $\chi_{DPP}(G) \leq \ww(G)+ 1 = sd^{(4)}(G)$ for every graph $G$.
%giving the inequality chain
%\[\ch(G) \leq \chi_P(G) \leq \chi_{DPP}(G) \leq sd^{(4)}(G).\]

Zhou, Zhu, and Zhu \cite{ZZZ} characterized the strict type-$3$ degeneracy of a graph in terms of \emph{removal schemes}, which are defined as follows.
%Recall that given a graph $G$ and a mapping %$f:V(G) \rightarrow \mathbb N$, 
%and given a vertex $u \in V(G)$ and a neighborhood subset $W \subseteq N(u)$, 
%the operation $\ds(G,f,u,W)$ returns the graph $G' = G - u$ and a mapping $f':V(G') \rightarrow \mathbb Z$ such that $f'(v) = f(v) - 1 $ for all $v \in N(u) \setminus W$
%and such that $f'$ agrees with $f$ at all other vertices of $G'$. Recall that the operation $\ds(G,f,u,W)$ is legal if and only if $f(u) > \sum_{w \in W} f(w)$ and $f'$ is positive-valued. 
%We say that $G$ is \emph{$f$-removable}, or that $(G,f)$ is removable, if all vertices of $G$ can be deleted 
%through repeated legal applications of 
%$\ds(G,f,\cdot,\cdot)$ and updating $G$ to the outputted graph $G'$ and $f$ to the outputted mapping $f'$.
%We say that $G$ is \emph{$k$-removable} if $G$ is $f$-removable for the constant function $f(v) = k$.
If $G$ is $ST^{(3)}$-$f$-degenerate for a function $f:V(G) \rightarrow \mathbb N$, then 
some sequence $\mathcal S$ of legal $\ds$ operations removes all vertices of $G$.
A sequence of legal $\ds$ operations on the graph $G$ with the initial function $f$ is called a \emph{(legal) removal scheme} on $(G,f)$. (As we only consider legal $\ds$ operations in this paper, we often omit the word ``legal.") We say that a removal scheme $\mathcal S$ is \emph{complete} if the operations of $\mathcal S$ remove all vertices of $G$.
Given a removal scheme $\mathcal S$ on $(G,f)$, we often define a function $\sv:V(G) \rightarrow 2^{V(G)}$ such that $\sv(u) = W$ for each vertex $u$ removed by $\mathcal S$, where $W \subseteq N(u)$ is the set that appears in the $\ds$ operation in $\mathcal S$ that removes $u$.

When we have a graph $G$ and a function $f:V(G) \rightarrow \mathbb N$, we often imagine that each vertex $v \in V(G)$ has a stack of $f(v)$ tokens.
Then, whenever we apply an operation $\ds(G,f,u,W)$ to $(G,f)$, we imagine that the vertex $u$ is removed from $G$, and then a single token is removed from each vertex $v \in N(u) \setminus W$. Then, $G$ is $ST^{(3)}$-$f$-degenerate
if and only if all vertices of $G$ can be removed by repeatedly applying $\ds$ operations in such a way that no vertex loses all of its tokens before being removed from $G$.

%When considering a pair $(G,f)$, we often imagine that each vertex $v \in V(G)$ has $f(v)$ tokens. Then, when we apply an operation $\ds(G,f,u,W)$ to $(G,f)$, we often imagine that the vertex $u$ is removed from $G$, and each remaining vertex of $N(u) \setminus W$ loses a single token. In this way, the final number of tokens at each vertex $v \in V(G)$ represents the value of $f'(v)$, where $f'$ is the function outputted by $\ds(G,f,u,W)$.
 %The tokens used in our interpretation of a removal scheme are similar to the tokens used in the online list coloring \cite{SchauzP, ZhuP} and online DP-coloring \cite{KKLZ} settings.

We note that a complete removal scheme $\mathcal S$ on $(G,f)$ uniquely defines a linear order $<$ on $V(G)$, given by the order in which the vertices of $G$ are deleted by the $\ds$ operations in $\mathcal S$. As discussed above, the complete removal scheme $\mathcal S$ also uniquely defines a function 
$\sv:V(G) \rightarrow 2^{V(G)}$ which maps each vertex $u \in V(G)$ to the subset $W \subseteq N(u)$
that $u$ saves when it is deleted.
We often identify a complete removal scheme $\mathcal S$ with the pair $(<, \sv)$ with which $\mathcal S$ uniquely corresponds.
%Indeed, the linear order $<$ shows the order in which the vertices of $G$ appear in the third argument of the $\ds$ operations in $\mathcal S$, and the $\sv$ function indicates the neighborhood $W \subseteq N(u)$ that appears in the fourth argument of the $\ds$ function after each $u \in V(G)$.

For some graph classes, the strict type-$4$ degeneracy parameter gives the best possible upper bound for choosability.
For instance, 
Han, Wang, Wu, Zhou, and Zhu \cite{ZhuWD} showed that if $G$ is a planar graph of girth at least $5$, then
$sd^{(4)}(G) \leq 3$, generalizing the classical theorem of Thomassen \cite{Thomassen3} stating that planar graphs of girth at least $5$ are $3$-choosable.

\subsection{Previous results}

For graphs of bounded chromatic number, Bernshteyn and Lee proved the following theorem.
\begin{theorem}[\cite{BL}]
\label{thm:wd_bip}
For each integer $r \geq 2$, there is an integer $d_0 = d_0(r)  \geq 1$ and a real number $\beta = \beta(r) >0$ such that if $G$ is a graph of maximum degree $d \geq d_0$ satisfying $\chi(G) \leq r$, then 
\[sd^{(4)}(G) \leq d - \beta \sqrt{d }.\]
\end{theorem}
They also proved that every $d$-regular graph $G$ on $n \geq 2$ vertices satisfies $sd^{(4)}(G) \geq d - \sqrt{2n} + 1$, 
which shows that Theorem \ref{thm:wd_bip} is tight in particular for dense regular bipartite graphs, up to the constant $\beta$.

%If $sd^{(3)}(G) \leq k$, then we say that $G$ is \emph{strict type-$3$ $k$-degenerate}.
%{\PB Do we need this last definition?}
%, and we write $\ww(G)$ for Bernshteyn and Lee's \emph{weak degeneracy} of $G$.

%The original weak$^*$ degeneracy definition of 
%Zhou, Zhu, and Zhu \cite{ZZZ} 
%also includes a $\vd_v$ operation 
%that deletes a vertex $v$ and then applies $\rv_u$ to all neighbors $u \in N(v)$. 
%However, 
 %we can delete all edges incident to $v$ with the same effect by first applying $\rv_v$ until obtaining a function $f$ for which $f(v) = 1$, and then applying $\ed_{uv}$ for all neighbors $u \in N(v)$. In this paper, we omit the $\vd_v$ operation for the sake of simplicity.

A simple inductive argument 
%using the recursive definition of strict type-$3$ degeneracy 
shows that
if $G$ is $ST^{(3)}$-$f$-degenerate, then $G$ is $L$-colorable for every $f$-assignment $L$ on $G$, implying that $ch(G) \leq sd^{(3)}(G)$.
Using a more sophisticated argument, Zhou, Zhu, and Zhu \cite{ZZZ} showed in fact that 
the strict type-$3$ degeneracy of a graph is an upper bound for its Alon-Tarsi number. 
They also showed that the strict type-$3$ degeneracy of a graph gives an upper bound for its DP-paintability.
% which we define fully in Section \ref{sec:parameters}.
These facts give the following inequality chains:
\begin{equation}
\label{eq:chain1}
    ch(G) \leq \chi_P(G) \leq  AT(G) \leq sd^{(3)}(G)  \leq sd^{(4)}(G) 
\end{equation}
\begin{equation}
\label{eq:chain2}
    ch(G) \leq \chi_P(G) \leq  \chi_{DPP}(G) \leq sd^{(3)}(G)  \leq  sd^{(4)}(G).
\end{equation}

By (\ref{eq:chain2}), Theorem \ref{thm:wd_bip} shows that if $G$ is an $r$-colorable graph of sufficiently large maximum degree $d$, then 
\begin{equation}
\label{eq:P-DP-UB}
\chi_P(G) \leq \chi_{DPP}(G) \leq d - \Omega(\sqrt{d})
\end{equation}
Based on current results, this upper bound for $\chi_P(G)$ is best known for all $r \geq 3$, and this upper bound for $\chi_{DPP}(G)$ is the best known for all $r \geq 2$. 

\subsection{Our results}
In this paper, we prove the following results, which improve the upper bounds for $\chi_P(G)$ and $\chi_{DPP}(G)$
given by (\ref{eq:P-DP-UB}).
%First, we show that if $G$ is an $r$-chromatic graph, then we have the following improved upper bound for $AT(G)$.
\begin{theorem}
\label{thm:AT}
Let $d \geq 0$ and $r \geq 2$ be integers.
If $G$ is a graph of maximum degree at most $d$ satisfying $\chi(G) \leq r$, then  \[\chi_P(G) \leq AT(G) \leq \left (1 - \frac{1}{4r+1} \right ) d + 2.\]
\end{theorem}

%Next, we prove the following theorem, which implies that an $r$-colorable graph of sufficiently large maximum degree $d$ satisfies $\chi_{DPP}(G) \leq d - \Omega (\sqrt{d \log d})$.

\begin{theorem}
\label{thm:chromatic_intro}
For each integer $r \geq 2$, there exists an integer $d_0 = d_0(r) \geq 1$ and a real number $\beta \geq \frac{1}{4000r}$
 such that the following holds for all values $d \geq d_0$.
If $G$ is a graph of maximum degree at most $d$ and chromatic number at most $r$, then
\[\chi_{DPP}(G) \leq sd^{(3)}(G) \leq d - \beta \sqrt{d\log d}.\]
\end{theorem}
As Bernshteyn and Lee \cite{BL} show that Theorem \ref{thm:wd_bip} is best possible for dense bipartite graphs, Theorem \ref{thm:chromatic_intro} implies the existence of a family of graphs $G$ for which the difference $sd^{(4)}(G) - sd^{(3)}(G)$ is arbitrarily large. Previously, even the existence of a single graph $G$ for which $sd^{(4)}(G) > sd^{(3)}(G)$ was unknown.

Our paper is organized as follows.
In Section \ref{sec:AT}, we prove Theorem \ref{thm:AT}.
Our proof uses a lemma of Esperet, Kang, and Thomass\'e \cite{EKT} stating that a graph of bounded chromatic number has an induced bipartite subgraph with large minimum degree.
%In Section \ref{sec:parameters}, we describe a sufficient condition for showing that a pair $(G,f)$ is strict type-$3$ degenerate in terms of \emph{removal schemes}.
Then, in Section \ref{sec:wd}, we prove Theorem \ref{thm:chromatic_intro}. In the first part of Section \ref{sec:wd}, we  prove the result for bipartite graphs; then, we use the same lemma of Esperet, Kang, and Thomass\'e \cite{EKT}  described above to finish the proof for $r \geq 3$.  We note that a similar approach gives an alternative (and perhaps simpler) proof of Theorem \ref{thm:wd_bip}.
We additionally show that the upper bound for $sd^{(3)}(G)$ in Theorem \ref{thm:chromatic_intro} cannot be improved apart from optimizing the value of $\beta$.

\section{A bound for paintability}
\label{sec:AT}
In this section, we prove Theorem \ref{thm:AT}, which
states that a graph $G$ of maximum degree $d$ and chromatic number at most $r$ has Alon-Tarsi number at most $\left (1 - \frac{1}{4r + 1} \right )d + 2$.
This result implies the same upper bound
for the paintability of $G$ and hence
gives an asymptotic improvement to the bound on $\chi_P(G)$ implied by Theorem \ref{thm:wd_bip}.
The proof of Theorem \ref{thm:AT} relies the following lemma of Esperet, Kang, and Thomass\'e \cite{EKT}.

\begin{lemma}
\label{lem:bipartite_subgraph}
    If $G$ is a graph of chromatic number $r$ and minimum degree $\delta$, then $G$ has an induced bipartite subgraph with minimum degree at least $\frac{\delta}{2r}$.
\end{lemma}

We also need the following folklore lemma.
\begin{lemma}
    Every graph $G$ has an orientation such that each vertex $v \in V(G)$ satisfies
    \[\deg^-(v) \in \{\lfloor \deg(v) / 2\rfloor, \lceil \deg(v) / 2 \rceil\}.\]
\end{lemma}
\begin{proof}
    Consider the graph $G'$ obtained from $G$ by adding a vertex $u$ adjacent to exactly those vertices $v \in V(G)$ for which $\deg_G(v)$ is odd. Then, $G'$ is Eulerian and hence has an Eulerian orientation. This orientation satisfies the property stated in the lemma even after deleting $u$.
\end{proof}

%\begin{proof}
 %   Let $\phi$ be a proper $k$ coloring of $G$. We choose two color classes $A,B$ from $G$ uniformly at random (with possible repetition), and we consider the subgraph $H$ induced by $A \cup B$. For each edge $e \in E(G)$, the probability that $e \in E(H)$ is $2/k^2$.
%\end{proof}

With these lemmas in place, we are ready to prove Theorem \ref{thm:AT}.

\begin{theorem13}
Let $d \geq 0$ and $r \geq 2$ be integers.
If $G$ is a graph of maximum degree at most $d$ satisfying $\chi(G) \leq r$, then  \[\chi_P(G) \leq AT(G) \leq \left (1 - \frac{1}{4r+1} \right ) d + 2.\]
\end{theorem13}
\begin{proof}
By (\ref{eq:chain1}), it is enough to show that $AT(G) \leq \left (1 - \frac{1}{4r+1} \right ) d + 2$.
We write $\epsilon = \frac{1}{4r+1}$, and we observe that $\epsilon = \frac{1-\epsilon}{4r}$. We consider a graph $G$ of maximum degree at most $d$ satisfying $\chi(G) \leq r$, and we aim to show that $AT(G) \leq (1-\epsilon)d + 2$.

We proceed by induction on $|V(G)|$.
We prove the stronger statement that each graph $G$ of maximum degree at most $d$ and chromatic number at most $r$ has an orientation $\vec G$ of maximum out-degree at most $(1-\epsilon)d+1$ with no directed odd cycle.
As every Eulerian subgraph of an oriented graph $\vec H$ is an edge-disjoint union of directed cycles, every odd Eulerian subgraph of $\vec H$ contains a directed odd cycle. Therefore, if $\vec G$ has no directed odd cycle, then $EO(\vec G) = 0$.
Then, as the edgeless subgraph of $\vec G$ is Eulerian, $EE(\vec G) \geq 1 > EO(\vec G)$, and hence $AT(G) \leq (1-\epsilon)d + 2$.

When $|V(G)| = 1$, the empty orientation has maximum out-degree $0$ and has no directed odd cycle;
therefore, 
the statement holds. 

Now, suppose that $|V(G)| \geq 2$, $\chi(G) \leq r$, and that $G$ has maximum degree at most $d$. %We claim that the following algorithm constructs an Alon-Tarsi orientation of $E(G)$ with maximum out-degree $(1 - \epsilon)d$.
If $G$ has degeneracy at most $(1-\epsilon)d + 1 $, then we consider an acyclic orientation $\vec G$ of $E(G)$ in which each vertex has out-degree at most $(1-\epsilon)d+1$. Clearly, the orientation $\vec G$ admits no directed odd cycle, so the statement holds.
On the other hand, if the degeneracy of $G$ is greater than $(1-\epsilon)d+1$, then $G$ has an induced subgraph $H$ with minimum degree at least $\lfloor (1-\epsilon)d \rfloor + 2 > (1-\epsilon) d$. By Lemma \ref{lem:bipartite_subgraph}, $H$ in turn has an induced bipartite subgraph $H'$ of minimum degree at least $\frac{(1-\epsilon)d }{2r} $.
We orient $E(H')$ so that each vertex $v \in V(H')$ satisfies \[\deg^-_{H'}(v) \geq \left \lfloor \frac{(1-\epsilon)d }{4r}   \right \rfloor >  \frac{(1-\epsilon)d }{4r}  - 1 = \epsilon d - 1.\]
We also give the orientation $uv$ to each edge $uv \in E(G)$ satisfying $u \in V(H')$ and $v \not \in V(H')$. 
Finally, by the induction hypothesis, we give $G \setminus H'$ an orientation of maximum out-degree $(1 - \epsilon)d+1$ with no directed odd cycle. 
Hence, we obtain an orientation $\vec G$ of $E(G)$.

We first check that $\vec G$ has maximum out-degree at most $(1-\epsilon)d+1$. Each vertex of $G \setminus H'$ has maximum out-degree at most $(1-\epsilon)d+1$ by the induction hypothesis. Furthermore, each vertex $u \in V(H')$ has in-degree at least 
$\epsilon d  - 1$ and hence has out-degree at most $ (1-\epsilon)d + 1$. 

We next check that $\vec G$ has no directed odd cycle. If $C$ is a directed odd cycle in $\vec G$, then as every edge in the cut $[V(H'), V(G) \setminus V(H')]$ is directed away from $H'$, $C$ has no edge in this cut. Furthermore, as $H'$ is bipartite, $C$ is not contained in $E(H')$, and hence $C$ has no edge of $H'$. Finally, by the induction hypothesis, $C$ is not contained in $E(G \setminus H')$. Thus, we conclude that a directed odd cycle $C$ in fact does not appear in $\vec G$. Therefore, $\vec G$ satisfies both of our desired properties, completing induction and the proof.
\end{proof}

We note that
unlike Theorems \ref{thm:wd_bip} and \ref{thm:chromatic_intro}, for each fixed $r  \geq 2$, the bound in Theorem \ref{thm:AT} holds for all values $d$, not just those values $d$ which are sufficiently large with respect to $r$. Hence, Theorem \ref{thm:AT} applies even to graph classes in which the chromatic number grows with the maximum degree. For instance, if $G$ is a triangle-free graph of maximum degree $d$, then $\chi(G) \leq (1+o(1)) \frac{d}{\log d}$ \cite{JohanssonTF, Molloy}. Hence, Theorem \ref{thm:AT} implies that $\chi_P(G) \leq AT(G) \leq d - (\frac{1}{4} + o(1)) \log d$. Similarly, if $G$ has large maximum degree $d$ and clique number $O(1)$, then $\chi(G) = O \left ( \frac{d \log \log d}{\log d} \right )$ \cite{Johansson, Molloy}, and thus $\chi_P(G) \leq AT(G) \leq d - \Omega \left ( \frac{\log d}{\log \log d} \right )$.

\section{A bound for DP-paintability}
\label{sec:wd}
In this section, we prove Theorem \ref{thm:chromatic_intro}, which implies that every graph of maximum degree $d$ and bounded chromatic number satisfies $\chi_{DPP}(G) \leq sd^{(3)}(G) \leq d - \Omega(\sqrt{d \log d})$. We first prove the result for bipartite graphs, and then we use Lemma \ref{lem:bipartite_subgraph} to extend our result to graphs of bounded chromatic number. 
We also show that the $\Omega(\sqrt{d \log d})$ term in the upper bound for $sd^{(3)}(G)$ is best possible.
% We use removal schemes as our main tool for upper bounding the strict type-$3$ degeneracy of graphs.

\subsection{Bipartite graphs}
We first consider bipartite graphs. We will need two probabilistic tools, namely the Lov\'asz Local Lemma and the Chernoff bound.
\begin{lemma}[\cite{LLL}]
\label{lem:LLL}
    Let $\mathcal B$ be a set of bad events. Suppose that each bad event $X \in \mathcal B$ occurs with probability at most $p$ and that each $X \in \mathcal B$ is independent with all but at most $D$ other events $X' \in \mathcal B$. If $pD < 1/e$, then with positive probability, no bad event in $\mathcal B$ occurs.
\end{lemma}
\begin{lemma}[{\cite[Chapter 4]{Mitzenmacher}}]
    \label{lem:chernoff}
    Let $X_1, \dots, X_n$ be independent Bernoulli variables, and let $\Pr(X_i = 1) = p$ for $i \in \{1,\dots,n\}$. Let $X = \sum_{i=1}^n X_i$, and let $\mu = pn$. Then, for $0 < \alpha < 1$, 
    \[\Pr( |X-\mu| > \alpha \mu) \leq 2 \exp \left ( - \frac 13 \alpha^2 \mu \right ).\]
\end{lemma}

We also need the following lemma, which appears in Bernshteyn and Lee \cite{BL} and easily follows from Hall's Marriage Theorem.
\begin{lemma}
\label{lem:star}
    Let $G$ be a bipartite graph with partite sets $B$ and $S$. Suppose that each vertex $b \in B$ satisfies $\deg(b) \leq q_1$ and each vertex $s \in S$ satisfies $\deg(s) \geq q_2$. Let $t \in \mathbb N$ satisfy $tq_1 \leq q_2$. 
    Then, there exists a spanning subgraph $H$ of $G$ satisfying the following:
    \begin{itemize}
        \item For each $b \in B$, $\deg_H(b) \in \{0,1\}$.
        \item For each $s \in S$, $\deg_H(s) = t$.
    \end{itemize}
\end{lemma}

The following theorem is our main tool for proving Theorem \ref{thm:chromatic_intro}.

\begin{theorem}
\label{thm:bip}
For all constants $0 < \epsilon \leq 1$, 
there exists an integer $d_0 \geq 1$ such that the following holds for all values $ d \geq d_0$.
Let $G$ be a bipartite graph with maximum degree $d$ and minimum degree at least $\epsilon d$, and let $f(v) = \deg(v) - \left \lfloor \frac{\epsilon}{1000} \sqrt{d \log d}\right \rfloor $ for each vertex $v \in V(G)$. Then, $G$ is $ST^{(3)}$-$f$-degenerate.
\end{theorem}
\begin{proof}
    Our proof contains a set of inequalities involving $d$, all of which hold when $d$ is sufficiently large. 
    We choose $d_0$ to be large enough so that these inequalities hold for all $d \geq d_0$.
    We often ignore floors and ceilings, as they do not affect our arguments. We imagine that each vertex $v \in V(G)$ initially receives $f(v)$ tokens, and we aim to construct a complete removal scheme for $(G,f)$.
    
Let $G$ have partite sets $A$ and $B$. 
We define $q = \lfloor d^{1/100} \rfloor$.
We also define several small constants, as follows:
\begin{eqnarray*}
\gamma &=& \frac{1}{2}\epsilon \\
\alpha &=& d^{-1/8} \\
p_S &=& \sqrt {\frac{\log d}{ d}} \\
p_0 &=& 1 - \gamma \\
p_1 &=& \frac{2}{3} \gamma \\
p_m &=& \frac{\gamma}{(m+1)(m+2)}\textrm{ for } m \in \{2,\dots, q\}.
\end{eqnarray*}

We observe that for $m \in \{1, \dots, q \}$, $p_1 + \dots + p_m = \gamma \left (1 - \frac{1}{m+2} \right )$, and $p_0 + p_1 + \dots + p_m = 1 - \frac{\gamma}{m+2}$.
 The following claim 
 designates a certain well-behaved subset $S \subseteq A$. The claim also
 states that $B$ can be partitioned into $q$ well-behaved sets $B_1, \dots, B_q$, along with a leftover set $B^*$.

\begin{claim}
\label{claim:degs}
    There
exists a subset $S \subseteq A$ and a partition $B = B_0 \cup  B_1 \cup \dots \cup B_{q} \cup B^*$ satisfying the following properties:
\begin{enumerate}
    %\item  \label{nInS} For each vertex $b \in B$, \[ (1-\alpha) \sqrt{d \log d } \leq  |N(b) \cap S| \leq (1+\alpha)  \sqrt{d \log d } ;\]
    \item  \label{nInS} For each vertex $b \in B$, \[ (1-\alpha) p_S \deg(b) \leq  |N(b) \cap S| \leq (1+\alpha)  p_S \deg(b) ;\]
    %\item \label{nInB0} For each vertex $a \in A$, \[1 \leq |N(a) \cap B_0| \leq 20 \log d;\]
%    \item \label{nInB0} For each vertex $a \in A$, \[1 \leq |N(a) \cap B_0| \leq 2p_0d;\]
%    \item \label{nInB1} For each vertex $a \in A$,     \[  (1-\alpha)\frac{2}{3} d \leq |N(a) \cap B_1| \leq (1+\alpha)\frac{2}{3} d ;\]
    %\item \label{nInBm} For $m \in \{2,\dots,r\}$ and each vertex $a \in A$,
    %\[ \frac{ 1-\alpha}{(m+1)(m+2)} d\leq |N(a) \cap B_m| \leq \frac{ 1+\alpha}{(m+1)(m+2)}d ; \]
    \item \label{nInBm} For each $m \in \{0,\dots,q\}$ and each vertex $a \in A$,
    \[ (1-\alpha) p_m \deg(a) \leq |N(a) \cap B_m| \leq (1+\alpha) p_m \deg(a) ; \]
    %\item \label{BstarSize} $|B^*| > n \sqrt{\frac{ \log d}{d}}$.
\end{enumerate}
\end{claim}
\begin{proof}
    We construct $S$ by adding each vertex $a \in A$ to $S$ independently with probability $p_S$. We construct our partition of $B$ by placing each vertex $b \in B$ into a single part as follows:
    \begin{itemize}
        \item For each $m \in \{0,\dots, q\}$, $b$ is placed in $B_m$ with probability $p_m$;
        \item $b$ is placed in $B^*$ with probability $1 - (p_0 + p_1 + \dots + p_{q})$.
    \end{itemize}
%    For each vertex $a \in A$,
   % let $X_{a,0}$ be the event that $|N(a) \cap B_0 | \not \in [1, 20 \log d]$.
    For each $a \in A$ and $m \in \{0,\dots,q\}$, let $X_{a,m}$ be the bad event that $|N(a) \cap B_m|$ is not in the interval specified in the claim. 
    For each vertex $b \in B$, let $Y_b$ be the bad event that $|N(b) \cap S|$ is not in the interval specified in the claim. 
    We observe that if no bad event $Y_b$ occurs, then Part (\ref{nInS}) of the claim holds, and 
    %If no bad event $X_{a,1}$ occurs, then (\ref{nInB1}) holds. 
    if for $m \in  \{0,\dots,q\}$ no bad event $X_{a,m}$ occurs, then Part (\ref{nInBm}) of the claim holds.
    Therefore, to show that subset $S \subseteq A$ and the partition of $B$ described in the claim exist, it suffices to show that all bad events can be avoided with positive probability.
%    If for $m \in \{0,1,\dots,r\}$ no bad event $X_{a,m}$ occurs, then 
 %   \begin{eqnarray*}
  %  \sum_{a \in A} \sum_{m =0}^r |N(a) \cap B_m| &\leq & 20 n \log d + dn (1+\alpha) \left ( \frac{2}{3} + \frac{1}{12} + \dots + \frac{1}{(r+1)(r+2)} \right ) \\
   % &=& n \left (20 \log d + d(1+\alpha)\frac{r+1}{r+2} \right )
   % \end{eqnarray*}
   % Since $G$ is $d$-regular, this implies that
    %\[\sum_{a \in A} |N(a) \cap B^*| \geq 
    %nd -  n \left (20 \log d + d(1+\alpha)\frac{r+1}{r+2} \right ) > n \left ( \frac{d}{r+2} - \alpha d - 20 \log d    
    %\right ) > nd \sqrt{\frac{\log d}{d}}.
    %\]
    %This sum counts each vertex in $B^*$ exactly $d$ times, so $|B^*| >  n \sqrt{\frac{\log d}{d}}.$ Hence, (5) holds.

    We use the Lov\'asz Local Lemma (Lemma \ref{lem:LLL}) to show that with positive probability, no bad event occurs. 
    %By Chernoff, the probability of a bad event $X_{a,0}$ is at most $\exp(-\frac{1}{3} \cdot 10 \log d) = d^{-10/3}$. Similarly, the probability of a bad event $X_{a,1}$ is at most $\exp(-\frac{1}{3} \alpha^2 \cdot \frac{2}{3} d) < d^{-4}$.
    We estimate the probability of each event $X_{a,m}$ using the Chernoff bound (Lemma \ref{lem:chernoff}) with $\mu = p_m \deg(a)$ and with our value $\alpha$.
    %for $m \in \{0,\dots,q\}$ and $a \in A$, 
    %letting $\mu = p_m \deg(a)$ implies that
    As $\deg(a) \geq \epsilon d$ and $p_m \leq p_q = \frac{\gamma}{(q+1)(q+2)}$,
    \[\Pr (X_{a,m} ) \leq   2 \exp \left (-\frac{1}{3} \alpha^2 \cdot p_m \deg(a) \right ) \leq 2
\exp \left (-\frac{1}{3} \alpha^2 \cdot \frac{\gamma \epsilon d}{(q+1)(q+2)} \right  ) <    
    d^{-4}.
    \]
    Similarly,
    given a vertex $b \in B$, letting $\mu = p_S \deg(b)$ in Lemma \ref{lem:chernoff} implies that
    \[\Pr(Y_b) \leq 2 \exp \left (-\frac{1}{3} \alpha^2  \cdot  p_S \deg(b)  \right ) \leq 2\exp \left (-\frac{1}{3} \alpha^2 \cdot \epsilon \sqrt{d \log d} \right ) < d^{-4}.\]
     %By the Lov\'asz Local Lemma, to show that all bad events are avoided with positive probability, it suffices to show that each bad event is independent with all but fewer than $d^3$ other bad events.

	Now, we count the dependencies of our bad events.
    Consider a bad event $Y_b$. The event $Y_b$ depends only on random choices made at neighbors of $b$; therefore, $Y_b$ is independent with all bad events except those events $Y_{b'}$ for which $b$ and $b'$ share a neighbor. The number of such $b'$ is at most $d^2$. Next, consider a bad event $X_{a,m}$. This bad event depends only on random choices made at neighbors of $a$. 
    Therefore, $X_{a,m}$ is independent with all other bad events except for those bad events $X_{a',m'}$ for which $a'$ and $a$ share a neighbor. The number of choices for such $a'$ is at most $d^2$, and the number of possible values for $m'$ is $q+1$. Therefore, $X_{a,m}$ is independent with all but at most $d^2(q+1) < d^3$ bad events.

    We have seen that each bad event occurs with probability at most $d^{-4}$, and each bad event is independent with all but fewer than $d^3$ other bad events. Since $ d^{-4} d^3 < 1/e$, the Lov\'asz Local Lemma (Lemma \ref{lem:LLL}) tells us that with positive probability, no bad event occurs. This completes the proof of the claim.
\end{proof}
%We observe that (\ref{nInS}) implies that $|S| \in [(1-\alpha)p_Sn , (1+\alpha) p_Sn]$.
Next, we show that for each part $B_m$ in our partition of $B$, there exists a certain matching-like structure between $S$ and $B_m$.

\begin{claim}
\label{claim:Hm}
    For each $m \in \{1,\dots,q\}$, there exists a spanning subgraph $H_m$ of $G[S \cup B_m]$ 
    and a constant $d_m = m \left  \lceil  \frac{\gamma p_m}{p_S} \right \rceil $
    such that $\deg_{H_m}(b) \leq m$ for each $b \in B_m$, and $d_{H_m}(s) = d_m$ for each $s \in S$.  
\end{claim}
\begin{proof}
    Let $m \in \{1,\dots,q\}$ be fixed.
    We construct $H_m$ through repeated applications of Lemma \ref{lem:star}.
    We prove the stronger claim that for each integer $0 \leq i \leq m$,
    there exists a spanning subgraph $H^i$ of $G[S \cup B_m]$ for which $\deg_{H^i}(b) \leq i$ for each $b \in B_m$, and $\deg_{H^i}(s) = i \lceil \frac{ \gamma p_m}{p_S} \rceil $ for each $s \in S$.

    The claim clearly holds for $i = 0$ by letting $H_m$ have no edges. Now, suppose that $1 \leq i \leq m$ and a suitable subgraph $H^{i-1}$ exists.
    By Claim \ref{claim:degs}, each vertex $b \in B_m$ has at most $q_1 := (1+\alpha)p_S d$ neighbors in $S$, and each vertex $s \in S$
    has at least 
    \[q_2 := (1-\alpha) \epsilon p_m  d - (i-1)\left \lceil \frac{\gamma p_m}{p_S}  \right \rceil \]
    neighbors in $B_m$ via the graph $G \setminus H^{i-1}$. Since $\frac{i p_m}{p_S} \leq \frac{mp_m}{p_S} < \sqrt d$ and $p_m d  > d^{9/10}$, it follows that $\frac{q_2}{q_1} > \frac{\epsilon p_m}{2 p_S} =  \frac{\gamma p_m}{p_S}$. Therefore, by Lemma \ref{lem:star}, we can find a star forest $F \subseteq G[S \cup B_m]$ whose edges are disjoint from $E(H^i)$
    so that each vertex in $S$ has degree exactly $\lceil  \frac{\gamma p_m}{p_S} \rceil$ in $F$, and each vertex in $B_m$ has degree $1$ or $0$ in $F$.
    Then, we 
    let $H^i = H^{i-1} \cup F$, and clearly $H^i$ satisfies the required properties. This completes induction, and the proof of the claim is complete by setting $H_m = H^m$.
\end{proof}
For $m \in \{1,\dots,q\}$, we define a graph $H_m$ as described in Claim \ref{claim:Hm}.
We also define the graph $H_0$ to be the graph on the vertex set $B_0 \cup S$ with no edges.
Finally, we define $H = H_0 \cup H_1 \cup \dots \cup H_{q}$.
Now, we create a removal scheme $\mathcal S = (<, \sv)$ on $G$ as follows. 
Iterating through $m = 0,1,\dots,q$ in order,
we delete the vertices $b \in  B_m$ one at a time, each time letting $\sv(b) = N_{H}(b)$. In other words, when we delete a vertex $b \in B_m$, we let $b$ save the vertices in $S$ which are adjacent to $b$ via $H_m$.

We first show that during this process, no vertex $a \in A$ loses all of its tokens. Indeed, by Claim \ref{claim:degs},
each vertex $a \in A$ has at most $(1+\alpha) (p_0 + \dots + p_{q}) \deg(a) = (1+\alpha) (1 - \frac{\gamma}{q+2}) \deg(a) $ neighbors in $B_0 \cup \dots \cup B_{q}$. Thus, after deleting all vertices in $B_0 \cup \dots \cup B_{q}$,
each $a \in A$ has at least 
\[f(a) -  (1+\alpha) \left  (1 - \frac{\gamma}{q+2} \right ) \deg(a) = \Omega(d^{0.99}) >0 \]
tokens remaining.

Next, we prove by induction on $m$ that for each $0 \leq m \leq q$ and vertex $b \in B_m$, our definition of $\sv(b)$ is legal. 
Recall that for each $b \in B_m$, $|\sv(b)| \leq m$.
Thus, when $m = 0$, each vertex $b \in B_0$ has a function $\sv(b) = \emptyset$, which is clearly legal.
Now, suppose that $1 \leq m \leq q$ and that $\sv(b')$ is legally defined for each vertex $b' \in B_0 \cup \dots \cup B_{m-1}$.
We observe that for each vertex $b \in B_m$,  as $\gamma = \frac 12 \epsilon$,
\[\frac{f(b)}{\gamma d/m} \geq \frac{\epsilon d - \frac{\epsilon}{1000} \sqrt{d \log d}}{\gamma d /m} = 2m - \frac{m}{500} \sqrt{\frac {\log d}{d} } > m \geq |\sv(b)|.\]
Thus, as long as each $s \in S$ has at most $\frac{\gamma d}{m}$ tokens when our removal scheme reaches $B_m$,
$\sv(b)$ is legally defined for each vertex $b \in B_m$.
Hence, we aim to show that after deleting $B_0 \cup \dots \cup B_{m-1}$, each vertex in $S$ has at most $\frac{\gamma d}{m}$ tokens.
By Claim \ref{claim:degs}, each vertex $s \in S$ satisfies 
\begin{equation}
\label{eq:sNLB}
|N(s) \cap (B_0 \cup \dots \cup B_{m-1})| \geq (1-\alpha) \deg(s) (p_0 + \dots + p_{m-1})
 = (1-\alpha)\deg(s) \left (1 - \frac{\gamma}{m+1} \right ).
\end{equation}
Furthermore, by the induction hypothesis, when our deletion scheme reaches $B_m$, each vertex $s \in S$ has been saved 
\begin{equation}
\label{eqn:svUB}
\sum_{i = 1}^{m-1} i \left \lceil  \frac{\gamma p_i}{p_S} \right \rceil < m^2 + \frac{\gamma}{p_S} \sum_{i = 1}^{m-1} p_i = m^2 + \frac{\gamma^2}{p_S} \left  ( 1 - \frac{1}{m+1} \right )
\end{equation}
times.
By combining (\ref{eq:sNLB}) and (\ref{eqn:svUB}), when our deletion scheme reaches $B_m$, the number of tokens that each $s \in S$ is at most 
\begin{eqnarray*}
& & \deg(s) - (1-\alpha) \deg(s)  \left (1 - \frac{\gamma}{m+1} \right ) + m^2 + \frac{\gamma^2}{p_S} \left ( 1 - \frac{1}{m+1} \right ) \\
&=& \frac{\gamma \deg(s)}{m+1} + \alpha \deg(s) \left (1 - \frac{\gamma}{m+1} \right ) + m^2 + \gamma^2 \sqrt{\frac{d}{\log d}} \left (1 - \frac{1}{m+1} \right ) \\
&=& \frac{\gamma \deg(s)}{m} - \frac {\gamma \deg(s)}{m(m+1)} + O(\alpha d) \\
&<& \frac{\gamma d}{m}.
\end{eqnarray*}
Hence, our functions $\sv(b)$ are legal for all $b \in B_m$. This completes induction.

After deleting $B_0 \cup \dots \cup B_{q}$, 
we delete each vertex $a \in A \setminus S$, each time letting $\sv(a) = \emptyset$. 
By Claim \ref{claim:degs}, each remaining vertex $b \in B^*$ has at least $(1-\alpha) p_S \deg(b) > \left \lfloor \frac{\epsilon}{1000} \sqrt{d \log d} \right \rfloor = \deg(b) - f(b)$ neighbors 
in $S$, so no vertex $b \in B^*$ runs out of tokens during this stage.

Next, we 
delete the vertices in $B^*$, defining $\sv(b) = \emptyset$ for each $b \in B^*$.
We argue that for each vertex $s \in S$, the number of tokens at $s$ is more than $|N(s) \cap B^*|$,
so that no vertex $s \in S$ loses all of its tokens. To prove this claim, we fix a vertex $s$, and we write $t(s)$ for the number of remaining tokens at $s$ at this stage. We aim to show that $t(s) - |N(s) \cap B^*| > 0$.
Since $s$ has been saved exactly $\sum_{m = 1}^{q} 
m \lceil  \frac{\gamma p_i}{p_S} \rceil$ times, we find that 
\begin{eqnarray}
\notag
t(s) - |N(s) \cap B^*|  &=& f(s) - |N(s) \cap (B_0 \cup \dots \cup B_{q})| + \sum_{m = 1}^{q} m \left \lceil \frac{\gamma p_m}{p_S} \right \rceil - |N(s) \cap B^*| \\
\notag
 &=& f(s) - \deg(s) + \sum_{m = 1}^{q} m \left \lceil  \frac{\gamma p_m}{p_S} \right \rceil \\
 \notag
 &=& \sum_{m = 1}^{q} m \left \lceil \frac{\gamma p_m}{p_S} \right \rceil - \left \lfloor \frac{\epsilon}{1000} \sqrt{d \log d} \right \rfloor  \\
 \notag
 &\geq & \frac{\gamma}{p_S} \sum_{m =1}^{q} mp_m
-  \frac{\epsilon}{1000} \sqrt{d \log d} 
  \\
  \label{eq:H-before}
& > &  \frac{\gamma}{3 p_S} \sum_{m=1}^{q} \frac{1}{m+2} - \frac{\epsilon}{1000} \sqrt{d \log d}.
\end{eqnarray}
Using the fact that $q = \lfloor d^{1/100} \rfloor$, we roughly estimate the sum in (\ref{eq:H-before}) as follows:
\begin{equation}
\label{eqn:H-estimate}
\sum_{m=1}^q \frac{1}{m+2}  = \sum_{i=1}^{q+2} \frac 1i - \frac 32 > \log q - \frac 32 =  \log\lfloor  d^{1/100} \rfloor  - \frac 32 > \frac{3}{400} \log d.
\end{equation}
By combining (\ref{eq:H-before}) and (\ref{eqn:H-estimate}), and using the fact that $\gamma = \frac 12 \epsilon$,
\begin{eqnarray*}
t(s) - |N(s) \cap B^*| 
&>& \frac{\gamma}{400 p_S} \log d - \frac{\epsilon}{1000} \sqrt{d \log d} \\
&=& \frac{\epsilon}{4000} \sqrt{d \log d} > 0.
\end{eqnarray*} 
Therefore, after deleting every vertex in $B^*$, each vertex in the remaining vertex set $S$
has at least one token.

Finally, since only the vertices in the independent set $S$ remain undeleted, we simply delete each vertex $s \in S$ and define $\sv(s) = \emptyset$.
Hence, we create a complete removal scheme $\mathcal S = (<, \sv)$ on $(G,f)$, and thus $G$ is $ST^{(3)}$-$f$-degenerate.
\end{proof}

Theorem \ref{thm:bip} has the following immediate corollary.

\begin{corollary}
\label{cor:bip}
    There exists a constant $d_0$ such that the following holds for all integers $d \geq d_0$. If $G$ is a bipartite graph of maximum degree $d$, then 
    \[sd^{(3)}(G) \leq d - \frac{1}{1000} \sqrt{d \log d}.\]
\end{corollary}
\begin{proof}
    Let $\epsilon = 1$, and then let $d_0$ be a sufficiently large value taken from Theorem \ref{thm:bip}. Suppose that $G$ is a bipartite graph of maximum degree $d \geq d_0$. We let $G'$ be a $d$-regular bipartite graph containing $G$. Then, by Theorem \ref{thm:bip}, $sd^{(3)}(G) \leq sd^{(3)}(G') \leq d - \frac{1}{1000} \sqrt{d \log d}$.
\end{proof}

Next, we
prove that $sd^{(3)}(K_{n,n}) >  n - \Omega (\sqrt{n \log n})$. 
This shows that the term of order $\sqrt{d \log d}$ in Theorem \ref{thm:bip} and Corollary \ref{cor:bip} is best possible, up to the multiplicative constant.
%We note that while Bernshteyn and Lee show a straightforward proof that $\ww(K_{n,n}) > n - 2 \sqrt{n}$, we will need to work somewhat harder to prove the correspond lower bound for $\rc(K_{n,n})$.
First, we need the following lemma.

\begin{lemma}
\label{lem:edge_LB}
Let $G$ be a graph, and let $f:V(G) \rightarrow \mathbb N$ be a function. If $G$ is $ST^{(3)}$-$f$-degenerate,
then $\sum_{v \in V(G)} f(v) > |E(G)|$.
\end{lemma}
\begin{proof}
    Zhou, Zhu, and Zhu \cite{ZZZ} showed that if $G$ $ST^{(3)}$-$f$-degenerate, then $G$ is $f$-Alon-Tarsi. Therefore, $G$ has an orientation such that for each vertex $v \in V(G)$, $\deg^+(v) < f(v)$. Then,
    \[|E(G)| = \sum_{v \in V(G)} \deg^+(v) < \sum_{v \in V(G)}  f(v).\]
    This completes the proof.
\end{proof}

%Now, by considering the complete bipartite graph $K_{n,n}$, we show that Theorem \ref{thm:chromatic_intro} cannot be improved in general using removal schemes.
\begin{theorem}
 $sd^{(3)}(K_{n,n}) > n - \sqrt{2 n (\log n+1)}$.
\end{theorem}
\begin{proof}
    We write $k = sd^{(3)}(K_{n,n})$.
    We write $A$ and $B$ for the two color classes of our graph $K_{n,n}$.
    %For the proof, it suffices to consider the minimum integer $k$ such that $K_{n,n}$ is $ST^{(3)}$-$f$-degenerate when $f:V(K_{n,n}) \rightarrow \{k\}$ is a constant function. 
    Since $K_{n,n}$ has an induced $C_4$, \cite[Theorem 1.6]{BL} implies that $k \leq sd^{(4)}(K_{n,n})  \leq n$. We write $j = n - k$.

    We imagine that each vertex $v$ in our graph $K_{n,n}$ begins with a stack of $f(v) = k$ tokens, and we let $\mathcal S = (<,\sv)$ be a complete removal scheme for our graph $K_{n,n}$.
    Without loss of generality, let $A$ be the partite set that is the first to lose $k$ vertices in the removal scheme $\mathcal S$.
    We let $A_0 \subseteq A$ denote the set of 
    $k$ vertices in $A$ which are deleted first in our removal scheme.
    When we consider a vertex $v$ at a point partway throughout our removal scheme, we let $t(v)$ denote the number of tokens remaining at $v$.

    We claim that $\sum_{a \in A_0} |\sv(a)| < (k-1)(\log k +1)$.
    We note that for $i \in \{0, \dots, k-1\}$, after the $i$th vertex of $ A_0$ is deleted, each vertex $b \in B$ satisfies $t(b) \geq k-i$, and furthermore $t(a) \leq k$. Therefore, if $a $ is the $(i+1)$st deleted vertex of $A_0$, then $|\sv(a)| (k-i) < k$, implying that $|\sv(a)| \leq \frac{k-1}{k-i}$. Therefore, 
    \begin{equation*}
        \sum_{a \in A_0} |\sv(a)| \leq \sum_{i=0}^{k-1} \frac{k-1}{k-i} = (k-1) \sum_{i=1}^k \frac 1i < (k-1)(\log k + 1).
    \end{equation*}
    A similar argument shows that if $B_0 \subseteq B$ is the set of vertices in $B$ which have been deleted when the last vertex of $A_0$ is removed, then $\sum_{b \in B_0} |\sv(b)| < (k-1)(\log k + 1)$.
    We write $s = \sum_{v \in A_0 \cup B_0} |\sv(v)|$ and observe 
    \begin{equation}
        \label{eq:s}
        s = \sum_{v \in A_0 \cup B_0} |\sv(v)| = \left ( \sum_{a \in A_0} |\sv(a)| \right ) + \left ( \sum_{b \in B_0} |\sv(b)| \right ) < 2(k-1)(\log k + 1).
    \end{equation}
    
    We consider the point in our removal scheme immediately after the last vertex of $A_0$ is deleted. At this point, the remaining vertices of $A$ form a set $A' = A \setminus A_0$ of exactly $n-k = j$ vertices, and the remaining vertices of $B$ form a set $B' = B \setminus B_0$ of $q > n-k = j$ vertices. 
    %For each remaining vertex $v \in A' \cup B'$, we write $t(v)$ for the number of tokens remaining at $v$. 
    By Lemma \ref{lem:edge_LB}, at this point in our removal scheme, 
    \begin{equation}
    \label{eqn:tv-sum}
        \sum_{v \in A' \cup B'} t(v) >  q j.
    \end{equation}

    We observe that when a vertex $a \in A_0$ is deleted from $K_{n,n}$, 
     each vertex $b \in B' \setminus \sv(a)$ loses exactly one token, so
    the vertices in $B'$ altogether lose $|B'| - |\sv(a)| = q - |\sv(a)|$ tokens. 
    Similarly, when a vertex $b \in B_0$ is deleted, 
    each vertex of $A' \setminus \sv(b)$ loses exactly one token, so
    the vertices in $A'$ altogether lose $|A'| - |\sv(b)| = j - |\sv(b)|$ tokens.
    Hence, the number of tokens removed from $A' \cup B'$ while deleting $A_0 \cup B_0$ is equal to $kq + (n-q)j - s$. As we have shown in (\ref{eq:s}) that $s < (2k-2)(\log k + 1)$, 
    it follows that the number of tokens removed from $A' \cup B'$ while deleting $A_0 \cup B_0$ 
    satisfies
    \begin{equation}
    \label{eqn:f-t}
        \sum_{v \in A' \cup B'} \left ( f(v) - t(v)  \right ) > kq + (n-q)j - (2k-2)(\log k + 1).
    \end{equation}
    As $\sum_{v \in A' \cup B'} f(v) =
    k(q + j)$,
    we combine (\ref{eqn:tv-sum}) and (\ref{eqn:f-t}) to obtain
    \[k(q+j) - q j > 
     kq + (n-q)j - (2k-2)(\log k + 1).\]
    Simplifying this, we see that 
    \[j(k-n) + (2k-2)( \log k + 1) > 0,\]
    or equivalently, $j^2 < (2k-2) (\log k+1)$. As $k \leq n$, we know that $j < \sqrt{(2n-2)( \log n + 1)}$. In other words, $k > n- \sqrt{(2n-2)( \log n+1)}$, completing the proof.    
\end{proof}

\subsection{$r$-chromatic graphs}
Now that we have proven an upper bound for the strict type-$3$ degeneracy of bipartite graphs, we are ready to prove Theorem \ref{thm:chromatic_intro}.

\begin{theorem14}
For each integer $r \geq 2$, there exists an integer $d_0 = d_0(r) \geq 1$ and a real number $\beta \geq \frac{1}{4000r}$
 such that the following holds for all values $d \geq d_0$.
If $G$ is a graph of maximum degree at most $d$ and chromatic number at most $r$, then
\[\chi_{DPP}(G) \leq sd^{(3)}(G) \leq d - \beta \sqrt{d\log d}.\]
\end{theorem14}
\begin{proof}
We show that the result holds when $\beta  = \frac{1}{4000 r}$ and $d_0$ is sufficiently large.
We set $\epsilon = \frac{1}{4r}$, 
and we choose a value $d_0$ that is sufficiently large for Theorem \ref{thm:bip} with respect to our choice of $\epsilon$.

We aim to prove that if $f(v) = d -\lfloor \beta \sqrt{d \log d} \rfloor $ for each $v \in V(G)$, then $(G,f)$ has a complete removal scheme.
For this, we use induction on $|V(G)|$.
If $|V(G)| = 1$, then as $d_0$ is sufficiently large, $f$ is positive-valued on $G$, and hence $(G,f)$ has a complete removal scheme.

Now, suppose that $|V(G)| \geq 2$.
We assign $f(v) = d -  \lfloor  \beta \sqrt{d \log d  } \rfloor   $ tokens to each vertex $v \in V(G)$.
If the degeneracy of $G$ is less than $\frac{1}{2}d $, then $sd^{(3)}(G) \leq sd(G) < \frac{1}{2}d  +1 < d - \lfloor \beta \sqrt{d \log d} \rfloor $, so the result holds.
Otherwise, 
%we let each vertex $v \in V(G)$ have an initial assignment $t(v) = d - \left \lceil c\sqrt{d \log d  }  \right  \rceil$ tokens.
as $G$ has degeneracy at least $\frac{1}{2}d $,
 $G$ has an induced subgraph of minimum degree at least $\frac{1}{2}d$. 
Using Lemma \ref{lem:bipartite_subgraph}, we take an induced bipartite subgraph $H$ of $G$ with minimum degree at least $\frac{d}{4r} $. Using the induction hypothesis, we let $\mathcal S'$ be a complete removal scheme of $G \setminus H$.

Now, consider the graph $H$ after carrying out the removal scheme $\mathcal S'$ on $G$. 
We see that each vertex $v \in V(H)$ satisfies $\epsilon d \leq \deg_H(v) \leq d$ and has at least $\deg_H(v) - \lfloor  \beta \sqrt{d \log d} \rfloor=  \deg_H(v) - \lfloor \frac{\epsilon}{1000}\sqrt{d \log d} \rfloor $ remaining tokens. Therefore, by Theorem \ref{thm:bip}, there exists a complete removal scheme $\mathcal S$ on $H$. We append $\mathcal S$ to $\mathcal S'$ to obtain a complete removal scheme for $(G,f)$. This completes induction and the proof.
\end{proof}

We conclude this section by noting that a similar application of Lemma \ref{lem:bipartite_subgraph} allows for an alternative proof of Bernshteyn and Lee's  \cite{BL} bound in Theorem \ref{thm:wd_bip}.
To prove Theorem \ref{thm:wd_bip}, 
one may first prove the statement for bipartite graphs and then extend the bound to $r$-chromatic graphs using a similar method to that of our proof of Theorem \ref{thm:chromatic_intro} above.

\section{Conclusion}
While we have succeeded in establishing several upper bounds on the values $\chi_P(G)$ and $\chi_{DPP}(G)$ for graphs $G$ of maximum degree $d$ and chromatic number $\chi(G)\leq r$, several questions still remain. First, it is natural to ask whether our upper bound of $\chi_P(G) \leq  \left (1 - \frac{1}{4r+1} \right ) d + 2$ in Theorem \ref{thm:AT} can be improved. Johansson \cite{JohanssonTF} and Molloy \cite{Molloy} observe that if $G$ is a graph satisfying $\chi(G[N(v)]) \leq r$ for each $v \in V(G)$, then 
\begin{equation}
\label{eqn:r-neighborhood}
ch(G) \leq O \left (\log r \frac{d}{\log d} \right ),
\end{equation}
implying the same upper bound for $r$-chromatic graphs. Therefore, it is reasonable to conjecture that a similar upper bound holds for paintability. Thus, we pose the following question.
\begin{question}
\label{q:sublinear}
    Let $r \geq 2$ be fixed, and let $G$ be a graph of maximum degree $d$ and chromatic number at most $r$. It is true that $\chi_P(G) = O \left (\frac{d}{\log d} \right )$?
\end{question}
One difficulty in answering Question \ref{q:sublinear} is that the current random approaches used to prove (\ref{eqn:r-neighborhood}), namely the Lov\'asz Local Lemma and entropy compression, do not translate easily to two-player games. These tools rely on the condition that an event at a vertex $v \in V(G)$ does not affect vertices far from $v$. However, this condition does not hold in two-player games, where a single game event may have a significant impact on a player's entire strategy. Pegden \cite{Pegden} discusses such limitations of the Lov\'asz Local Lemma in two-player games.

Next, it is unclear whether our upper bound on the Alon-Tarsi number of an $r$-chromatic graph in Theorem \ref{thm:AT} is tight. This leads us to our next question.
\begin{question}
\label{q:AT}
    What is the maximum Alon-Tarsi number of a graph of maximum degree $d$ and chromatic number at most $r$?
\end{question}
When $r=2$, the answer to Question \ref{q:AT} is $\lceil d/2 \rceil + 1$, as every orientation of a bipartite graph is an Alon-Tarsi orientation, and the maximum out-degree of an orientation is minimized by an almost Eulerian orientation. Hence, we see that Theorem \ref{thm:AT} is not tight for $r=2$.
For $r \geq 3$, the tightness of Theorem \ref{thm:AT} is open.

\section{Acknowledgment}
We are grateful to Tao Wang for carefully reading an earlier draft of this paper and pointing out an error in our terminology, which has helped us greatly improve the writing quality. We are also grateful for the helpful comments of the anonymous referees, whose feedback also helped us improve this paper.

\raggedright
\bibliographystyle{plain}
\bibliography{bib}

\end{document}